\documentclass[review]{elsarticle}

\usepackage{lineno,hyperref}
\usepackage{amsmath,amssymb}
\usepackage{graphicx}
\usepackage{mathtools}
\usepackage{tikz}
\usepackage{subfigure}
\usepackage{mathptmx}
\usepackage{latexsym}
\usepackage{amsthm}
\newtheorem{theorem}{Theorem}
\newtheorem{assumption}[theorem]{Assumption}
\newtheorem{definition}[theorem]{Definition}

\newtheorem{lemma}[theorem]{Lemma}
\newtheorem{corollary}[theorem]{Corollary}
\newtheorem{proposition}[theorem]{Proposition}

\modulolinenumbers[5]

\journal{Journal of \LaTeX\ Templates}

\begin{document}

\begin{frontmatter}
\title{On the Well-posedness of a Nonlinear Fourth-Order Extension of Richards' Equation}


\author[mymainaddress]{Alaa Armiti-Juber \corref{mycorrespondingauthor}}
\cortext[mycorrespondingauthor]{Corresponding author}
\ead{alaa.armiti@mathematik.uni-stuttgart.de}

\author[mymainaddress]{Christian Rohde}

\address[mymainaddress]{Institute for Applied Analysis and Numerical Simulation, University of Stuttgart, Pfaffenwaldring 57, 70569 Stuttgart, Germany}

\begin{abstract}
We study a nonlinear fourth-order extension of Richards' equation that describes infiltration processes in unsaturated soils. We prove the well-posedness of the fourth-order equation by first applying Kirchhoff's transformation to linearize the higher-order terms. The transformed equation is then discretized in time and space and a set of a priori estimates is established. These allow, by means of compactness theorems, extracting a unique weak solution. Finally, we use the inverse of Kirchhoff's transformation to prove the well-posedness of the original equation. 
\end{abstract}


\end{frontmatter}

\section{Introduction}
\label{intro}
The process of fluid infiltration through unsaturated soil is an important part of the hydrological cycle as it represents many crucial examples, such as the flow of rain water or waste fluids into water aquifers and the flow of salt-water into coastal aquifers. These infiltration processes are usually described using Richards' model \cite{Bear91}. Recent experiments on fluid infiltration show that, even in homogeneous porous media, an initially planar front does not remain planar. The fluids infiltrate in preferential flow paths taking the shape of fingers with different widths and velocities. As most of the fluid channelizes in the fingers with high velocity, this may have crucial effects on the environment as it reduces the time needed for a contaminant to reach the underground water. Experiments show also that constant flux infiltration into homogeneous porous media leads to higher saturation at the wetting front than behind the front. This natural behavior is called saturation overshoots and is believed to cause the gravity-driven fingering \cite{Juanes,DiCarlo2013}. 

Richards' model is unable to describe saturation overshoots, because it is a second-order parabolic differential equation fulfilling the maximum principle. Moreover, it is unable to predict fingered flows, as nonlinear stability analysis shows that the model is unconditionally stable \cite{Egorov2003,Nieber2005}. Therefore, many approaches have been suggested to modify Richards' model \cite{CF,HG1993,Schweizer2012}. 

In this paper, we propose a nonlinear fourth-order extension of Richards' equation. This extension is related to the fourth-order model in \cite{CF}, while having the benefit that both second- and fourth-order terms can be simultaneously linearized using Kirchhoff's transformation, which is more convenient for the well-posedness analysis later.

We prove in this paper the well-posedness of the proposed nonlinear fourth-order extension of Richards' equation. The paper has the following structure: Section \ref{sec:modeling} presents Richards' equation and our proposed nonlinear fourth-order extension. In Section \ref{sec:4th-preliminaries}, Kirchhoff's transformation is applied to the fourth-order model as a preparation step for the analysis in the following section, then a list of assumptions is provided. In Section \ref{sec:4th-wellposedness}, we prove the well-posedness of the transformed fourth-order model. In Section \ref{sec:4th-regularity}, we improve the regularity of the weak solution. Finally, we prove the well-posedness of the nonlinear fourth-order model in Section \ref{sec:4th-original}. 

\section{Modeling in Unsaturated Soil}
\label{sec:modeling}
This section presents two models that describe fluid flows in unsaturated soils: the classical Richards' model and a nonlinear fourth-order extension of it.

\subsection{Richards' model}
We consider a bounded domain $\Omega \subset \mathbb{R}^3$ in the zone of unsaturated soil, where gas occupies most of the pores. Since gas in this zone is naturally connected to the atmospheric air, its pressure is constant and equals the atmospheric air pressure. Assuming that water infiltrates through the domain $\Omega$ under the effect of gravity and capillary forces, the two-phase flow model for the infiltrating water is a combination of the mass conservation equation and Darcy's law 
\begin{align}
\begin{array}{r l}
\phi\partial_{t} S+\nabla\cdot\textbf{v}&=0,\vspace{5pt}\\
\textbf{v}&=-K_f(S) \left(\dfrac{\nabla p}{\rho g}-\textbf{e}_3\right),
\end{array}
\label{eq:one-phase-flow}
\end{align}
respectively. Here, $S=S(\textbf{x},t)\in [0,1]$ is saturation, $\textbf{v}=\textbf{v}(\textbf{x},t)\in\mathbb R^3$ is averaged velocity and $p=p(\textbf{x},t)\in\mathbb R$ is pressure of the infiltrating water phase. The porosity $\phi$ is assumed to be constant, $\rho=1$ is water density, $g$ is the gravitational acceleration, and $\textbf{e}_3=(0,0,1)^T$. We also consider the closure relation 
\begin{align}
 p_c=p_g-p,
\end{align}
where $p_g=p_{\text{air}}$ is constant. Then, using the van Genuchten parameterization \cite{Genuchten1980} of the capillary pressure $p_c=p_c(S)$, equation \eqref{eq:one-phase-flow} simplifies to \textbf{Richards' equation}
\begin{equation}
 \phi\partial_{t}S+\nabla\cdot\left(K_{f}(S) \left( \textbf{e}_3+\frac{\nabla p_c(S)}{g}\right)\right)=0.
 \label{eq:Richards}
\end{equation}

\subsection{The Nonlinear Fourth-Order Extension}
\label{sec:4th-order-model}
We propose a fourth-order extension of Richards's equation \eqref{eq:Richards} by adding a third-order regularizing term to Darcy's equation, i.e.
\begin{align}
 \textbf{v}=K_{f}(S)\nabla\left(z+\frac{1}{g}\,p_{c}(S)\right)- \frac{\epsilon}{g}\,\nabla\Big(\nabla \cdot \big(K_f(S)\nabla p_c(S)\big)\Big),
\label{eq:darcy2}
\end{align}
where $\epsilon$ is a small parameter. Substituting \eqref{eq:darcy2} into the continuity equation in \eqref{eq:one-phase-flow} yields the nonlinear fourth-order model
\begin{align}
  \partial_{t}S+\nabla\cdot \Bigl(K_{f}(S)\big(\textbf{e}_3+ \frac{1}{g}\nabla p_c(S)\big)\Bigr)-\frac{\epsilon}{g}\Delta\nabla\cdot\Bigl( K_f(S)\nabla p_c(S)\Bigr)=0.
\label{eq:Armiti1}
\end{align}
\begin{figure}
\centering
\begin{tikzpicture}[scale = 0.8,>=latex]
\draw[->] (-3,0) -- (2,0) node[anchor=west] {$p$};
\draw[->] (0,-0.55) -- (0,3) node[anchor=east] {$S$} ;
\draw[thick, color=blue] (-3,0.1)--(-1.4,0.1) ;
\draw[thick, color=blue] (-1.4,0.1) parabola[bend at start] (-0.7,1);
\draw[thick, color=blue] (-0.7,1) parabola[bend at end] (0,2);
\draw[thick, color=blue] (0,2)--(2,2) ;
\draw  (0,2.1) node[anchor=east] {$1$};
\end{tikzpicture}
\hspace{1.5cm}
\begin{tikzpicture}[scale = 0.8,>=latex]
\draw[->] (-1,0) -- (3,0) node[anchor=west] {$S$};
\draw		(2,0) node[anchor=north] {1};
\draw		(0,2) node[anchor=east] {1};
\draw[->] (0,0) -- (0,3) node[anchor=east] {$K_f$} ;
\draw[thick, color=blue] (0,0) parabola[bend at start] (2,2);
\draw[dashed] (2,0) -- (2,2);
\end{tikzpicture}
\caption{Water saturation $S$ as a function $p:=-\frac{p_c}{g}$ (left). Conductivity $K_f$ as a function of $S$ (right).}
\label{fig:saturation and konductivity}
\end{figure}
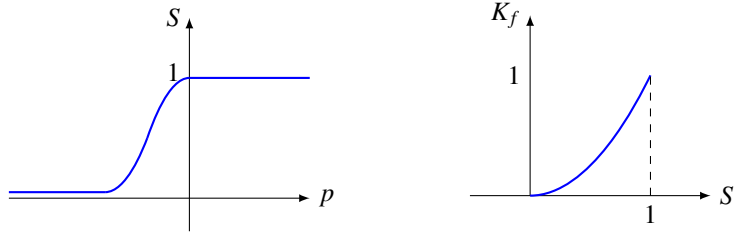
Since capillary pressure $p_c$ is a strictly monotone decreasing function of saturation $S$, its inverse is well-defined. Thus, we can write saturation $S$ as an increasing function of $p\coloneqq -\frac{p_c}{g}$ such that
\begin{align*}
 S(p)=\left\{ \begin{array}{cc}
               S(-\frac{p_c}{g}),&\quad p\leq 0,\\
               1, & \quad p>0,
              \end{array}\right.
\end{align*}
as shown in Figure \ref{fig:saturation and konductivity}. The Figure shows also the conductivity $K_f=K_f(S(p))$, which is a monotone increasing function of $S$. Using the inverse function $p$, the fourth order model \eqref{eq:Armiti1} can be written as
\begin{equation}
  \partial_{t}S(p)+ \nabla\cdot\Big( K_f\big(S(p)\big)\textbf{e}_3\Big) -\nabla\cdot \Bigl(K_{f}\big(S(p)\big)\nabla p\Bigr)+\gamma\Delta\nabla\cdot\Bigl( K_f\big(S(p)\big)\nabla p\Bigr)=0,
\label{eq:Armiti2}
\end{equation}
in $\Omega\times (0,T)$ with pressure $p$ is the unknown and $\gamma\coloneqq \tfrac{\epsilon}{g}$. Since we are interested in the existence of weak solutions in the space $L^2(0,T;H_0^2(\Omega))$, equation \eqref{eq:Armiti2} is augmented with the initial and boundary conditions
\begin{align}
  \begin{array}{r c c}
	p(.,0) & =  p^0 & \quad\text{ in } \Omega,\\
	p & = 0  & \quad\text{ on }\partial \Omega\times [0,T], \\
 	\nabla p\cdot \textbf{n} &= 0 & \quad\text{ on }\partial \Omega\times [0,T],
	\end{array}
 	\label{eq:4th-IBC}
\end{align}
where $\textbf{n}$ is the outer normal vector at the boundary $\partial \Omega$.

\section{Preliminaries and Assumptions}
\label{sec:4th-preliminaries}
In this section, we apply Kirchhoff's transformation to the fourth-order model \eqref{eq:Armiti2} to linearize the second- and the fourth-order terms. Then, we summarize all assumptions that are required throughout the paper.  

Kirchhoff's transformation is a continuous monotone increasing map defined as
	\begin{equation*}
        \psi :=\left\{
	\begin{array}{l} 
	\mathbb{R}\rightarrow\mathbb{R}	\\
	p\mapsto \psi(p)= \int_{0}^{p} K_{f}(S(\tau))d\tau 
	\end{array}\right.
	\end{equation*}
where $\psi(p)$ is the transformed pressure. We set $u:=\psi(p)$. Then, as Figure \ref{fig:Kirchhoff} shows, we have $u=p$ for $p\geq 0$, because $K_{f}(S(p))=1$. Moreover, there exists a lower bound $u_{l}<0$ of $u$ such that $u_{l}\coloneqq \lim_{p\rightarrow -\infty} \psi(p)=-\int_{-\infty}^0  K_{f}(S(p))\,dp$. In other words, the lower bound $u_{l}$ equals the area under the graph of $K_f$ multiplied by $-1$.

Applying the Leibniz rule on the transformed pressure $u$ gives
\begin{align}
\begin{array}{rl}
 \nabla u&=\,K_f \big(S(p)\big)\nabla p,\\
\Delta u&=\,\nabla\cdot \big(K_{f}(S(p))\nabla p\big),\\
\partial_t u&=\,K_f \big(S(p)\big)\partial_t p.
\end{array}
\label{eq:chainrule}
\end{align}
As the inverse function $\psi^{-1}:(u_{l},\infty)\rightarrow \mathbb{R}$ is well-defined, we define the function
	\begin{equation*}
	 b(u):=S(\psi^{-1}(u)),
	\end{equation*}
such that  
\begin{equation*}
 b^{\prime}(u)=\dfrac{S^{\prime}(p)}{K_{f}(S(p))}.
\end{equation*} 
Then, the transformed fourth-order model is given as:	
	\begin{equation}
	\partial_{t}b(u) + \nabla\cdot\Bigl( K_f( b(u))\textbf{e}_3\Bigr)- \Delta u + \gamma\Delta^2 u=0,
	\label{eq:linear1}
	\end{equation}
with the transformed initial and boundary conditions 
	\begin{equation}
	\begin{array}{r l c}
	u(.,0)&=u^0 & \text{in } \Omega\times \lbrace 0 \rbrace, \\
	u&=0 & \text{  on } \partial \Omega\times (0,T),\\
	\nabla u\cdot \textbf{n}  &=0 & \text{  on } \partial \Omega\times (0,T).\\
	\end{array}
	\label{eq:IBC}
	\end{equation}

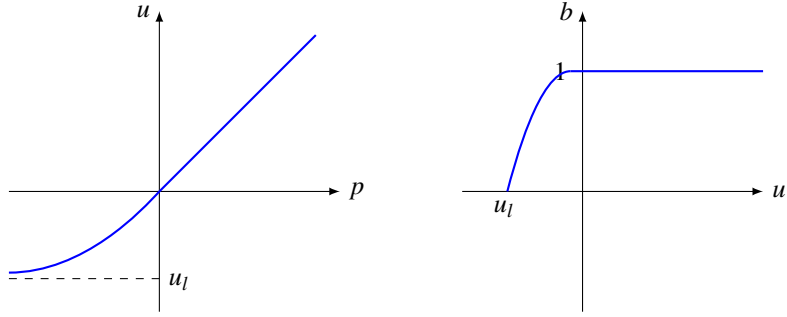
\begin{figure}
\centering
\begin{tikzpicture}[scale = 0.8,>=latex]
\draw[->] (-2.5,0) -- (3,0) node[anchor=west] {$p$};
\draw[->] (0,-2) -- (0,3) node[anchor=east] {$u$} ;

\draw[thick, color=blue]((0,0) parabola[bend at end] (-2.5,-1.35)   ;
\draw[thick, color=blue] (0,0) -- (2.6,2.6);
\draw[dashed] (-2.5,-1.45)--(0,-1.45);
\draw (0,-1.5) node[anchor=west] {$u_{l}$};
\end{tikzpicture}
\hspace{1.0cm}
\begin{tikzpicture}[scale = 0.8,>=latex]
\draw[->] (-2,0) -- (3,0) node[anchor=west] {$u$};
\draw	(0,0)
		(-1.3,0) node[anchor=north] {$u_{l}$};
\draw		
		(-0.1,2.01) node[anchor=east] {1};
\draw[->] (0,-2) -- (0,3) node[anchor=east] {$b$} ;
\draw[thick, color=blue] (-0.2,2) parabola[bend at start] (-1.25,0);
\draw[thick, color=blue] (-0.2,2)--(3,2) ;
\end{tikzpicture}
\caption{Transformed pressure $u=\psi(p)$ (left) and transformed saturation $b(u)$ (right).}
\label{fig:Kirchhoff}
\end{figure}

For $\gamma=0$, the wellposedness of of \eqref{eq:linear1} is proved in \cite{Alt,Merz2010}. The wellposedness of other fourth-order parabolic equations describing thin film growth is investigated in \cite{Alvarez2015,Duan2016,Liu2007,Sandjo2015}. Due to the nonlinearity of the first term on the left side of equation \eqref{eq:linear1}, we follow \cite{Alt} and define the Legendre transform $B$ for the primitive of $b$, 
	\begin{equation}
	B:=\left \{
	\begin{array}{l}
	\mathbb{R}\rightarrow\mathbb{R}^{+}\\
	z\mapsto B(z)=\int_{0}^{z}b(z)-b(s) ds,
	\end{array}\right.
	\label{eq:legendre}
	\end{equation}
The map $B$ satisfies the following properties: 
	\begin{lemma}
 	 If $b$ is a continuous and monotone increasing function, then the Legendre transform $B$, defined in \eqref{eq:legendre}, satisfies
	\begin{equation}
	B(z)-B(z_{0})\geq \Bigl(b(z)-b(z_{0})\Bigr)z_{0},\nonumber
	\end{equation}
	 \begin{equation}
	B(z)-B(z_{0})\leq \Bigl(b(z)-b(z_{0})\Bigr)z,\nonumber
	\end{equation}
	for any $z, z_{0}\in \mathbb{R}$.
	\label{lem:Bproperties}
	\end{lemma}
	\begin{proof}
	The continuity and the monotonicity of $b$ imply the existence of a convex function $\phi\in C^{1}(\mathbb{R},\mathbb{R})$ such that $b=\phi':=\dfrac{d\phi}{du}$. 
	The definition of $B$ and the property that $b=\phi'$ give
	\begin{equation}
	 B(z)=\int_{0}^{z}(b(z)-\phi'(s))ds=b(z)z-\Bigl(\phi(z)-\phi(0)\Bigr).
	\end{equation}
	Then, we have
	\begin{equation}
	B(z)-B(z_{0})=b(z)z-b(z_{0})z_{0}-\Bigl(\phi(z)-\phi(z_{0})\Bigr)\nonumber.
	\end{equation}
	To prove the first inequality, we add $\pm b(z)z_{0}$ to the right side of the above equation, then we have
	\begin{equation}
	B(z)-B(z_{0})=\Bigl(b(z)-b(z_{0})\Bigr)z_{0}\underbrace{-b(z)(z_0-z)-\Bigl(\phi(z)-\phi(z_{0})\Bigr)}_{M \coloneqq}\nonumber.
	\end{equation}
	The Taylor expansion and the convexity of $\phi$ imply that $M > 0$, which proves the inequality. The second inequality follows similarly by adding $\pm b(z_0)z$.
	\end{proof}
	
	We summerize all assumptions that are required throughout the paper:
\begin{assumption}  \begin{enumerate}
  \item The domain $\Omega \subset \mathbb{R}^{3} $ is an open bounded connected region with boundary $\partial \Omega\in C^5$ and $0< T < \infty$.
  \item The initial condition $u^{0}\in H^2_0(\Omega)$ satisfies $u^0,~b(u^{0}),~B(u^0)$ $\in L^{\infty}(\Omega)$.
  \item The function $b:(u_{l},\infty)\rightarrow (0,1] $ is strictly positive, monotone increasing and Lipschitz continuous.
  \item The conductivity function $K_f:(u_{l},\infty)\rightarrow (0,1]$ is Lipschitz continuous, strictly positive, and there exists a constant $ \beta >0$ such that, for all $z\in\mathbb{R}$, the following growth condition holds
  \begin{equation*}
 \Bigl(K_f(b(z))\Bigr)^2\leq \beta \Bigl(1+B(z)\Bigr).
\end{equation*}
 \end{enumerate}
\label{ass:fourth-order}
\end{assumption}

 \section{Well-posedness of the Transformed Fourth-Order Model} 
\label{sec:4th-wellposedness}
In this section, we prove the well-posedness of the transformed fourth-order model \eqref{eq:linear1} with the initial and boundary conditions \eqref{eq:IBC}. In section \ref{sec:apprximatemodel}, we approximate the time derivative in the model using backward differences producing a series of elliptic equations. Then, we apply Galerkin's method to these equations and prove the existence of weak solutions for the discrete problem. In Section \ref{sec:4th-a priori}, we prove a set of a priori estimates on the sequence of discrete solutions. These are used in Section \ref{sec:4th-convergence} to conclude a weak convergence of the sequence. Then, we prove that the limit is a weak solution for the transformed problem. Finally, we prove in Section \ref{sec:4th-uniqueness} the uniqueness of the weak solution.

\subsection{An Approximate Model}
\label{sec:apprximatemodel}
Let $N>0$ be an integer and $h=T/N$. Approximating $\partial_t b(u)$ in \eqref{eq:linear1} using the backward difference $\tfrac{b(u\cdot,t))-b(u(\cdot,t-h))}{h}$ yields for almost all $t\in[0,T]$ the biharmonic equation
\begin{equation}
\frac{b(u(\cdot,t))-b(u(\cdot,t-h))}{h}+\nabla\cdot \Bigl(K_f(b(u(\cdot,t)))\textbf{e}_3\Bigr)-\Delta u(\cdot,t)+\gamma\Delta^{2} u(\cdot,t) =0.
\label{eq:elliptic1}
\end{equation}

For any arbitrary but fixed $t\in [0,T]$, we consider weak solutions of \eqref{eq:elliptic1} in the Hilbert space $V(\Omega)= H_{0}^{2}(\Omega)$. Let $\{w_i\}_{i\in\mathbb{N}}$ be a countable orthonormal basis of $V$. By applying Galerkin's method to equation \eqref{eq:elliptic1}, the solution space $V(\Omega)$ is projected into a finite dimensional space $V_M(\Omega)$ spanned by a finite number of the orthonormal functions $w_i,~i=1,...,M$. For $h>0$ and a positive integer $M$, we search the coefficients $\alpha^h_{Mi} \in L^{\infty}((0,T)),\, i=1,\dots M$ defining the function
        \begin{equation}
        u^h_M(t):=\sum_{i=1}^{M}\alpha^{h}_{Mi}(t)w_{i}.
	\label{eq:discretesolution}
        \end{equation}
These coefficients are chosen such that, for almost all $t\in[0,T]$, the equation
	\begin{align}
        \frac{1}{h}\int_{\Omega}\Bigl(b\big(u^h_M(t)\big)-b\big(u^h_M(t-h)\big)\Bigr)w_i \,d\textbf{x}+& \int_{\Omega}\nabla u^h_M(t)\cdot\nabla w_i+\gamma\Delta u^h_M(t)\Delta w_i \,d\textbf{x}\nonumber\\=&\int_{\Omega} K_f\big(b(u^h_M(t))\big)\textbf{e}_3\cdot\nabla w_i\,d\textbf{x}
	\label{eq:approx.Eq}
        \end{align}
holds for all $i=1,\cdots,M$. The discrete initial condition is defined as 
        \begin{equation}
         u^{h}_M(t)=u_M^0,\quad \text{ for }t \in ( -h,0] ,
         \label{eq:approx.IC}
        \end{equation}
where, $u_M^0$ is the $L^2$-projection of the initial data $u_0$ into the finite dimensional space $V_M(\Omega)$.

To prove the existence of solutions for the discrete problem \eqref{eq:approx.Eq} and \eqref{eq:approx.IC}, we need the below stated technical lemma on the existence of zeros of a vector field \cite{Evans}.
	
\begin{lemma}
Let $r>0$ and $\textbf{v}:\mathbb{R}^n\rightarrow\mathbb{R}^n$ be a continuous vector field, which satisfies $\textbf{v}(\textbf{x})\cdot \textbf{x}\geq 0$ if $|\textbf{x}|=r$. Then, there exists a point $\textbf{x}\in B(0,r)$ such that $\textbf{v}(\textbf{x})=\textbf{0}$.
\label{lem:Bvectorfield}
\end{lemma} 

\begin{lemma}
For any $M\in\mathbb{N}$, $h>0$, and almost any $t\in[0,T]$, let $u^{h}_M(t-h) \in V_M(\Omega) $ and  
\begin{align}
 h \leq \frac{1}{\beta},
 \label{eq:ass-delta t}
\end{align}
where $\beta>0$ is given as in Assumption \ref{ass:fourth-order}(4). Then, equation \eqref{eq:approx.Eq} has a solution $u^h_M(t)\in V_M(\Omega)$.
\end{lemma}
\begin{proof}
We define the vector field $\textbf{f}:\mathbb{R}^{M}\rightarrow\mathbb{R}^{M}$ such that $\textbf{f}=(f_1,\cdots,f_M)$ and the vector $\alpha_M^h=(\alpha_{M1}^h,\cdots,\alpha_{MM}^h)$ of the unknown coefficients of $u^h_M(t)$ in equation \eqref{eq:discretesolution}. Then, we have 
	\begin{align}
	f_i(\alpha_M^h):=&~\frac{1}{h}\int_{\Omega}\left(b(u^h_M(t))-b(u^h_M(t-h)) \right)w_i\,d\textbf{x}-\int_{\Omega}K_f\big(b(u^h_M(t))\big)\textbf{e}_3\cdot\nabla w_i\,d\textbf{x}\nonumber\\&+\int_{\Omega}\left(\nabla u^h_M(t)\cdot\nabla w_i+\gamma\Delta u^h_M(t)\Delta w_i\right)\,d\textbf{x},
	\label{eq:premitive}
	\end{align}
for all $ i=1,\cdots,M$. Note here that $u^h_M(t-h)$ for $t\in(0,h]$ is well-defined by the choice of the initial condition \eqref{eq:approx.IC}. Using Assumption \ref{ass:fourth-order}(3) and \ref{ass:fourth-order}(4), the vector field $\textbf{f}$ is continuous. Moreover, we have
	\begin{align}
	\textbf{f}(\alpha_M^h)\cdot\alpha_M^h=&~\dfrac{1}{h}\int_{\Omega}\left(b(u^h_M(t))-b(u^h_M(t-h)\right)u^h_M(t)\,d\textbf{x}-\int_{\Omega}K_f\big(b(u^h_M(t)\big)\textbf{e}_3\cdot\nabla u^h_M\,d\textbf{x}\nonumber\\
	&+ \int_{\Omega}\nabla u^h_M\cdot\nabla u^h_M\,d\textbf{x}+\gamma\int_{\Omega}\Delta u^h_M\Delta u^h_M\,d\textbf{x}.\nonumber
	\end{align}
Applying Lemma \ref{lem:Bproperties} on the first term of the right side and Cauchy's inequality on the second term yield
	\begin{align*}
	\textbf{f}(\alpha_M^h)\cdot\alpha_M^h \geq&~\dfrac{1}{h}\int_{\Omega}\left(B(u^h_M(t))-B(u^h_M(t-h))\right)\,d\textbf{x}- \dfrac{1}{2}\int_{\Omega}\left(K_f(b(u^h_M(t))\right)^2\,d\textbf{x}\\ &+\frac{1}{2}\int_{\Omega}\vert\nabla u^h_M\vert^{2}\,d\textbf{x}+\gamma\int_{\Omega}(\Delta u^h_M)^{2}\,d\textbf{x}.\nonumber
	\end{align*}
	The growth condition in Assumption \ref{ass:fourth-order}(4), equation \eqref{eq:discretesolution}, and the orthonormality of the basis functions $w_i,\,i=1,\cdots,M$, imply that
	\begin{align*}
	\textbf{f}(\alpha_M^h)\cdot\alpha_M^h\geq&~ \dfrac{1}{h}\int_{\Omega}\left(B(u^h_M(t))-B(u^h_M(t-h))\right)\,d\textbf{x}-\beta\int_{\Omega}\left(1+B(u^h_M(t))\right)\,d\textbf{x}\\
	&+\int_{\Omega}\left( \dfrac{1}{2}\left| \sum_i^M\alpha_{Mi}^h \nabla w_i \right|^{2}+ \dfrac{\gamma}{2}\left(\sum_i^M\alpha_{Mi}^h \Delta w_i \right)^{2}\right)\,d\textbf{x},\\\geq& ~\left(\dfrac{1}{h}-\beta\right)\int_{\Omega}B(u^h_M(t))\,d\textbf{x}-\left(\beta|\Omega|+\dfrac{1}{h}\int_{\Omega}B\left(u^h_M(t-h)\right)\,d\textbf{x}\right)\\&+\left(\dfrac{1}{2}+\dfrac{\gamma}{2}\right)\left|\alpha_M^h \right|^2.
	\end{align*}
	The first term of the right side of above inequality is nonnegative using condition \eqref{eq:ass-delta t}. Noting that $u^h_M(t-h)\in V_M(\Omega)$ is known and setting $r=|\alpha_M^h(t)|$ yields that $\textbf{f}(\alpha_M^h(t))\cdot\alpha_M^h(t)\geq 0$ provided that $r$ is large enough. Thus, Lemma \ref{lem:Bvectorfield} implies the existence of a vector $\alpha_M^h(t)\in\mathbb{R}^M$ satisfying $\textbf{f}(\alpha_M^h)=0$. Now, using \eqref{eq:premitive}, we obtain the existence of a function $u^h_M(t)$ that satisfies the discrete equation \eqref{eq:approx.Eq}.
\end{proof}

\subsection{A Priori Estimates}
\label{sec:4th-a priori}
We proved already the existence of a sequence $\{S_M^h\}_{M}\in\mathbb{N},h>0\subset V_M(\Omega)$ of discrete solution of the discrete problem \eqref{eq:approx.Eq} and \eqref{eq:approx.IC}. In the following, we prove a set of a priori estimates on the sequence that are essential for the convergence analysis in the next subsection.
\begin{lemma}
There exists a constant $c>0$ such that \emph{ 
\begin{align*}
\emph{ess}\sup_{t\in[0,T]}\int_{\Omega}(B(u^h_M(t))\,d\textbf{x}+\int_{0}^{T}\int_{\Omega} \vert\nabla u^h_M\vert^{2}+\gamma( \Delta u^h_M)^{2}\,d\textbf{x}\,dt \leq c\int_{\Omega}B(u_M^0)\,d\textbf{x},
\end{align*} }
for all $h>0$ and $M\in\mathbb{N}$.
\label{lem:apriori1}
\end{lemma}
\begin{proof}
Multiplying equation \eqref{eq:approx.Eq} by $\alpha_{Mi}^h$, summing for $i=1,\cdots,M$, and then integrating from $0$ to an arbitrary time $\tau\in[0,T]$ yields
\begin{align}
 \dfrac{1}{h} \int_{0}^{\tau} \int_{\Omega} &\left(b(u^h_M(t))-b(u^h_M(t-h)) \right) u^h_M(t)\,d\textbf{x}\,dt+\int_{0}^{\tau}\int_{\Omega}|\nabla u^h_M|^2\,d\textbf{x}\,dt\nonumber\\&+\gamma\int_{0}^{\tau}\int_{\Omega}(\Delta u^h_M)^2\,d\textbf{x}\,dt = \int_{0}^{\tau}\int_{\Omega}K_f\big(b(u^h_M\big)\textbf{e}_3\cdot \nabla u^h_M \,d\textbf{x}\,dt.
 \label{eq:apriori1-1}
\end{align}
Applying the first inequality in Lemma \ref{lem:Bproperties} to the first term on the left side of equation \eqref{eq:apriori1-1} and Cauchy's inequality to the right side yield
\begin{align*}
\dfrac{1}{h} \int_{0}^{\tau}\int_{\Omega} \left(B(u^h_M(t))\,-\right.&\left. B(u^h_M(t-h)\right)\,d\textbf{x}\,dt+ \int_{0}^{\tau}\int_{\Omega}|\nabla u^h_M|^2+\gamma(\Delta u^h_M)^2\,d\textbf{x}\,dt\nonumber\\\leq & ~\dfrac{1}{2}\int_{0}^{\tau}\int_{\Omega}\left(K_f(b(u^h_M)\right)^2 \,d\textbf{x}\,dt+\dfrac{1}{2}\int_{0}^{\tau}\int_{\Omega}|\nabla u^h_M|^2\,d\textbf{x}\,dt.\nonumber
\end{align*}	
Applying the growth condition in Assumption \ref{ass:fourth-order}(4) to the first term on the right side of the above equation gives
	\begin{align*}
	\dfrac{1}{h} \int_{0}^{\tau}\int_{\Omega} \left(B(u^h_M(t))-\,\right.&\left. B(u^h_M(t-h))\right)\,d\textbf{x}\,dt+ \frac{1}{2 }\int_{0}^{\tau}\int_{\Omega}|\nabla u^h_M|^2\,d\textbf{x}\,dt\nonumber \\ &+\gamma\int_{0}^{\tau}\int_{\Omega}(\Delta u^h_M)^2\,d\textbf{x}\,dt \leq \beta \int_{0}^{\tau}\int_{\Omega}\left(1+B(u^h_M(t))\right)\,d\textbf{x}\,dt.
	 \end{align*}
Applying summation by parts to the first term on the left side of the above equation, and noting that $u^h_M$ is a step function in time, leads to
	\begin{align}
	\int_{\Omega}B(u^h_M(\tau))\,d\textbf{x}+\dfrac{1}{2}\int_{0}^{\tau}\int_{\Omega} \vert\nabla u^h_M\vert^{2}&\,d\textbf{x}\,dt +\gamma\int_{0}^{\tau} \int_{\Omega}\vert\Delta u^h_M\vert^{2}\,d\textbf{x}\,dt\nonumber\\ \leq & \beta|\Omega|T +\int_{\Omega}B(u_M^0)\,d\textbf{x} +\beta\int_{0}^{\tau}\int_{\Omega}B(u^h_M(t))\,d\textbf{x}\,dt.\nonumber
	\end{align}
	Note that $B(u^h_M)$ is nonnegative and summable on $[0,T]$, where the summability results from substituting $z_0=0$ into the second inequality in Lemma \ref{lem:Bproperties}, the boundedness of $b$, and the choice that the coefficients $\alpha_{M,i}^h\in L^{\infty}((0,T))$. Hence, Gronwall's inequality is applicable and implies the existence of a constant $c>0$ depending on $\beta,\,|\Omega|$, and $T$ such that
	\begin{equation*}
	\emph{ess}\sup_{t\in[0,T]}\int_{\Omega}\left(B(u^h_M(t)\right)\,d\textbf{x}+\int_{0}^{T}\int_{\Omega}\left(\dfrac{1}{2} \vert\nabla u^h_M\vert^{2}+\gamma( \Delta u^h_M)^{2}\right)\,d\textbf{x}\,dt \leq c\int_{\Omega}B(u_M^0)\,d\textbf{x}.
	\end{equation*}
	\end{proof}
\begin{corollary}
It holds that $ u_M^h\in L^{2}(0,T;H_0^2(\Omega))$ for all $M\in\mathbb{N}$ and $h>0$.
\label{cor:poincare}
\end{corollary}
\begin{proof}
Lemma \ref{lem:apriori1} and Poincar\'{e}'s inequality imply the existence of a constant $C>0$ such that
\begin{align*}
 \int_0^{T}\int_{\Omega} (u^h_M)^2\,d\textbf{x}\,dt \leq C,
\end{align*}
for all $M\in\mathbb{N}$ and $h>0$. Moreover, the biharmonic operator $Lu:\Delta u+ \Delta^2 u$ can be written as a combination of two second-order elliptic operators 
\begin{align*}
 L_1 w=&\Delta w+w,\\
 L_2 u=&\Delta u.
\end{align*}
Hence, the basis functions $w_i$ of the biharmonic operator $L$ can be chosen as a combination of the eigenfunctions of the operators $L_1$ and $L_2$. These eigenfunctions belong to the space $C^{3}(\Omega)$, whenever the boundary $\partial \Omega \in C^{5}$, \cite{Evans}. Hence, using Gauss' theorem and Cauchy's inequality, we obtain
	\begin{align*}
 	\int_{0}^{T}\int_{\Omega}(\partial_{x_ix_j}u^h_M)^{2}\,d\textbf{x}\,dt=&
 	-\int_{0}^{T}\int_{\Omega}\partial_{x_i}u^h_{M} \partial_{x_ix_jx_j}u^h_M\,d\textbf{x}\,dt=
 	\int_{0}^{T}\int_{\Omega}\partial_{x_ix_i}u^h_M \partial_{x_jx_j}u^h_M \,d\textbf{x}\,dt\\ \leq& \,\frac{1}{2}\int_{0}^{T}\int_{\Omega}(\partial_{x_ix_i}u^h_M )^2 \,d\textbf{x}\,dt+\frac{1}{2}\int_{0}^{T}\int_{\Omega}(\partial_{x_jx_j}u^h_M)^2 \,d\textbf{x}\,dt\\=&\,\frac{1}{2}\Vert \Delta u^h_M\Vert_{L^2(\Omega\times(0,T))}.
	\end{align*}
for all $i,\,j\in\{1,\cdots, d\}$. Thus we have  $D^{\sigma} u \in L^{2}(\Omega\times(0,T))$ for all index vectors $\sigma\in\mathbb{N}\times\mathbb{N}\times\mathbb{N}$ with $|\sigma|=2$ . 
\end{proof}
In the following lemma, we prove an a priori estimate on the backward difference quotient $\tfrac{b(u^h_M(t))-b(u^h_M(t-h))}{h}$.
\begin{lemma}
There exists a constant $c>0$ such that\emph{
\begin{align*}
\frac{1}{h}\int_{0 }^{T}\int_{\Omega}\left(b(u^h_M(t))-b(u^h_M(t-h))\right)\phi\,d\textbf{x}\,dt\leq c,
\end{align*}}
for any $\phi\in L^2(0,T;V_M(\Omega))$, $M\in\mathbb{N}$ and $h>0$.
\label{lem:apriori2}
\end{lemma}
\begin{proof}
Let $m\leq M$ be a positive integer and choose a function $\phi\in L^{\infty}(0,T;H_0^2(\Omega))$ such that for almost all $t\in[0,T]$
\begin{align}
 \phi(t)=\sum_{i=1}^m \alpha_{Mi}^h(t)w_i,
\end{align}
where $\alpha_{Mi}^h\in L^{\infty}((0,T))$, $i=1,\cdots,m$, are given functions and $w_i\in H_0^2(\Omega)$, $i=1,\cdots,m$, belong to the orthonormal basis of the subspace $V_M(\Omega)$. Multiplying equation \eqref{eq:approx.Eq} by $\alpha_{Mi}^h(t)$, summing for $i=1,...,M$, and then integrating from $0$ to $T$ yields  
\begin{align*}
	\dfrac{1}{h}\int_{0}^{T}\int_{\Omega}\left(b(u^h_M(t))- b(u^h_M(t-h))\right)\phi(t)\,d\textbf{x}\,dt=&\int_{0}^{T}\int_{\Omega}K_f(b(u^h_M(t))\textbf{e}_3\cdot \nabla \phi(t) \,d\textbf{x}\,dt\\&- \int_{0}^{T}\int_{\Omega}\nabla u^h_M(t)\cdot \nabla \phi(t)\,d\textbf{x}\,dt \\&-\gamma\int_{0}^{T}\int_{\Omega}\Delta u^h_M(t) \Delta \phi(t)\,d\textbf{x}\,dt.
	\end{align*}
Applying Cauchy's inequality on the terms on the right side of the above equation then using the growth condition in Assumption \ref{ass:fourth-order}(4) gives
	\begin{align*}
	 \dfrac{1}{h}&\int_{0}^{T}\int_{\Omega} \left(b(u^h_M(t))- b(u^h_M(t-h))\right)\phi(t)\,d\textbf{x}\,dt \\ \leq & \dfrac{1}{2}\int_{0}^{T}\int_{\Omega} \left(K(b(u^h_M(t))\right)^2\,d\textbf{x}\,dt+ \int_{0}^{T}\int_{\Omega}|\nabla \phi(t)|^2\,d\textbf{x}\,dt+ \dfrac{1}{2}\int_{0}^{T}\int_{\Omega}|\nabla u^h_M|^2\,d\textbf{x}\,dt\\ &+ \dfrac{\gamma}{2}\int_{0}^{T}\int_{\Omega}\Delta\phi(t)^2\,d\textbf{x}\,dt+ \dfrac{\gamma}{2}\int_{0}^{T}\int_{\Omega}(\Delta u^h_M)^2\,d\textbf{x}\,dt\\ \leq&  \dfrac{\beta}{2}\int_{0}^{T}\int_{\Omega}\left(1+B(u^h_M(t)\right)\,d\textbf{x}\,dt+ \int_{0}^{T}\int_{\Omega}|\nabla \phi(t)|^2\,d\textbf{x}\,dt+\dfrac{1}{2}\int_{0}^{T}\int_{\Omega}|\nabla u^h_M|^2\,d\textbf{x}\,dt\\ &+ \dfrac{\gamma}{2}\int_{0}^{T}\int_{\Omega}\Delta\phi(t)^2\,d\textbf{x}\,dt+\dfrac{\gamma}{2}\int_{0}^{T}\int_{\Omega}(\Delta u^{h}_{
M})^2\,d\textbf{x}\,dt.
	\end{align*}
Then, Lemma \ref{lem:apriori1} and the choice that $\phi\in L^{\infty}(0,T;H_0^1(\Omega))$ implies the existence of a constant $c>0$ such that 
	\begin{equation*}
	 \dfrac{1}{h}\int_{0}^{T}\int_{\Omega}\left(b(u^h_M(t)-b(u^h_M(t-h))\right)\phi(t)\,d\textbf{x}\,dt\leq c.
	\end{equation*}
\end{proof}

\begin{corollary}
There exist constants $\delta_0,\,c>0$ such that\emph{
\begin{align*}
\frac{1}{\delta}\int_{\delta }^{T}\int_{\Omega}\left(b(u^h_M(t))-b(u^h_M(t-\delta))\right)\left(u^h_M(t)-u^h_M(t-\delta)\right)\,d\textbf{x}\,dt\leq c,
\end{align*}}
for any $M\in\mathbb{N}$, $h>0$ and $\delta\in(0,\delta_0)$.
\label{cor:apriori2}
\end{corollary}
\begin{proof}
Choosing $\phi=u^h_M(t)-u^h_M(t-h)$ in Lemma \ref{lem:apriori2} yields 
\begin{align*}
\frac{1}{h}\int_{0 }^{T}\int_{\Omega}\left(b(u^h_M(t))-b(u^h_M(t-h))\right)\left(u^h_M(t)-u^h_M(t-h)\right)\,d\textbf{x}\,dt\leq c.
\end{align*}
Noting that $u^h_M$ is a step function in time, we obtain
	\begin{eqnarray*}
	\frac{1}{\delta}\int_{0 }^{T}\int_{\Omega}\left(b(u^h_M(t))-b(u^h_M(t-\delta))\right)\left(u^h_M(t)-u^h_M(t-\delta)\right)\,d\textbf{x}\,dt\leq c,
	\end{eqnarray*}
for any $\delta>0$ such that $\vert \delta - h \vert$ is small enough.
	\end{proof}

\subsection{Convergence Results}
\label{sec:4th-convergence}
In this subsection, we show the convergence of the sequence $\{u^h_M\}_{M\in\mathbb{N},h>0}$ of discrete solutions of equation \eqref{eq:approx.Eq} to a weak solution of the transformed fourth-order problem \eqref{eq:linear1} and \eqref{eq:IBC}. This result is summarized in Theorem \ref{thm:existence}. The proof of the theorem depends on the a priori estimates in Section \ref{sec:4th-a priori} and the following proposition by Alt and Luckhaus \cite{Alt}.

\begin{proposition}[Alt and Luckhaus \cite{Alt}]
Assume that $z_{\epsilon}\rightharpoonup z$ in $L^{2}(0,T;H^{1}(\Omega))$ as $\epsilon\rightarrow 0$ and there exists a constant $C>0$ such that \emph{
\begin{equation}
\dfrac{1}{\delta}\int_{0}^{T-\delta}\int_{\Omega}\bigl( b(z_{\epsilon}(t+\delta)-b(z_{\epsilon}(t)\bigr)\bigl(z_{\epsilon}(t+\delta)-z_{\epsilon}(t)\bigr)\,d\textbf{x}\,dt\leq C,
\label{eq:estimate2}
\end{equation}}
holds for any small $\delta>0$ and \emph{
\begin{align*}
 \int_{\Omega} B(z_{\epsilon}(t))\,d\textbf{x}\leq C,\quad\text{ for } 0<t<T.
\end{align*}}
Then, $ b(z_{\varepsilon})\rightarrow b(z)$ in $L^{1}(\Omega\times(0,T))$ and $B(z_{\epsilon})\rightarrow B(z)$ almost everywhere.
\label{prop:Alt1}
\end{proposition}

Before we state and prove the first main theorem in this chapter, we remind that the Sobolev space $L^2(0,T;H_0^2(\Omega))$ and its dual $L^{2}(0,T;H^{-2}(\Omega))$ are equipped with the norms
\begin{align*}
\Vert u\Vert_{L^2(0,T;H^2(\Omega))} =& \int_0^T\int_{\Omega}\left( u^2+|\nabla u|^2+ |D^2 u|^2\right)\,d\textbf{x}\,dt,\\
\Vert L\Vert_{L^2(0,T;H^{-2}(\Omega))} =&\sup\left\{ L(u)\mid  u\in L^2(0,T;H_0^2(\Omega)),\, \Vert u\Vert_{L^2(0,T;H^2(\Omega))} \leq 1\right\}.
\end{align*}
In addition, we state Cauchy's inequality that will be repeatedly used throughout the coming sections
\begin{align}
 ab\leq \epsilon a^2 + \frac{b^2}{4\epsilon}\quad\quad \forall a,\,b\in \mathbb{R},\, \epsilon>0.
 \label{eq:cauchy}
\end{align}

\begin{theorem}
Let Assumption \ref{ass:fourth-order} be satisfied and $h\leq \tfrac{1}{\beta}$. Then, problem \eqref{eq:linear1}, \eqref{eq:IBC} has a weak solution $u\in L^2(0,T;H_0^2(\Omega))$ that satisfies
	\begin{enumerate}
	\item $K_f(b(u))\in L^{2}(\Omega\times(0,T))$, $\partial_t b(u)\in L^{2}(0,T;H^{-2}(\Omega))$, and\emph{
	\begin{equation}
	\int_{0}^{T}\int_{\Omega} \Bigl(\partial_{t}b(u)\phi- K_f(b(u))\textbf{e}_3 \cdot\nabla\phi+ \nabla u \cdot\nabla\phi+\gamma\Delta u\Delta \phi\Bigr) \, d\textbf{x}\,dt=0,
	\label{eq:4th-condition1}
	\end{equation}}
	for every test function $\phi \in L^{2}(0,T;H_{0}^{2}(\Omega))$.
	\item $b(u)\in L^{\infty}(0,T;L^1(\Omega))$, $\partial_t b(u)\in L^{2}(0,T;H^{-2}(\Omega))$, and \emph{
	\begin{equation}
	\int_{0}^{T}\int_{\Omega} \partial_{t}b(u)\phi  \, d\textbf{x}\,dt=\int_{0}^{T}\int_{\Omega}\bigl(b(u)-b^{0}\bigr)\partial_{t}\phi \, d\textbf{x}\,dt, \label{eq:4th-condition2}
	\end{equation}}
	holds for all test functions $\phi\in L^{2}(0,T;H_{0}^{2}(\Omega))$ with $\partial_t \phi\in L^{1}(0,T; L^{\infty}(\Omega))$ and $\phi(\cdot,T)=0$.
	\end{enumerate}
	\label{thm:existence}
	 \end{theorem}
\begin{proof}
Using Corollary \ref{cor:poincare} and the Weak Compactness theorem, there exists a function $u\in L^{2}(0,T;H_0^2(\Omega))$ such that, up to a subsequence, 
\begin{equation}
 u^h_M\rightharpoonup u\quad\text{in }L^{2}(0,T;H_0^2(\Omega)),
 \label{eq:4th-conv1}
\end{equation}
as $M\rightarrow \infty$ and $h\rightarrow 0$. The next step in the proof is to show that the function $u\in L^2(0,T; H_0^2(\Omega))$ fulfills the conditions \eqref{eq:4th-condition1} and \eqref{eq:4th-condition2}. Thus, we consider an arbitrary test function $\phi\in L^2(0,T;V_m(\Omega))$ such that for a fixed integer $m$ and for almost all $t\in(0,T)$ is given as
\begin{align}
 \phi(t)=\sum_i^m \alpha_{i}^h(t) w_i,
\end{align}
where $\alpha_{i}^h\in L^{\infty}(0,T)$, $i=1,\cdots,m$, are given functions and $w_i\in H_0^2(\Omega)$, $i=1,\cdots,m$, belong to the orthonormal basis of the subspace $V_m(\Omega)$. Choosing $m<M$, multiplying equation \eqref{eq:approx.Eq} by $\alpha_{i}^h(t)$, summing for $i=1,\cdots, m$, and then integrating with respect to time yields
\begin{align}
\dfrac{1}{h} \int_{0}^{\tau} \int_{\Omega} &\left(b(u^h_M(t))-b(u^h_M(t-h)) \right) \phi(t)\,d\textbf{x}\,dt+\int_{0}^{\tau}\int_{\Omega}\nabla u^h_M\cdot\nabla\phi\,d\textbf{x}\,dt\nonumber\\&+\gamma\int_{0}^{\tau}\int_{\Omega}\Delta u^h_M\Delta \phi\,d\textbf{x}\,dt = \int_{0}^{\tau}\int_{\Omega}K_f(b(u^h_M)\textbf{e}_3\cdot \nabla \phi \,d\textbf{x}\,dt.
\label{eq:4th-weakdiscreate}
\end{align}

In the following we show that equation \eqref{eq:4th-weakdiscreate} converges as $m\rightarrow \infty$ and $h\rightarrow 0$ to equation \eqref{eq:4th-condition1}. The weak convergence \eqref{eq:4th-conv1}, Corollary \ref{cor:apriori2}, and Proposition \ref{prop:Alt1} imply the strong convergences,
\begin{align}
  b(u^h_M)\rightarrow b(u)\quad \text{ in } L^{1}(\Omega\times(0,T)),
  \label{eq:4th-conv2}
\end{align}
and 
\begin{align*}
  B(u^h_M)\rightarrow B(u)\quad \text{ almost everywhere. }
\end{align*}
The strong convergence of $ B(u^h_M)$ and the estimate in Lemma \ref{lem:apriori1} leads to
\begin{align}
 B(u)\in L^{\infty}(0,T;L^1(\Omega)).
  \label{eq:4th-conv3}
\end{align}
Hence, Assumption \ref{ass:fourth-order}(2) and the first inequality in Lemma \ref{lem:Bproperties} with $z_0=u^0$ imply
\begin{align}
 b(u)\in L^{\infty}(0,T;L^1(\Omega)).
   \label{eq:4th-conv33}
\end{align}
The Lipschitz continuity of the flux function and the strong convergence \eqref{eq:4th-conv2} imply
\begin{align*}
   K_f(b(u^h_M))\rightarrow K_f(b(u))\quad\text{ in } L^{1}(\Omega\times(0,T)),
\end{align*} 
and consequently, we have
\begin{align}
   K_f(b(u^h_M))\rightarrow K_f(b(u))\quad \text{ almost everywhere. }
  \label{eq:4thh-conv4}
\end{align}
However, we need to prove at least a weak convergence of $ K_f(b(u^h_M))$ in $L^{2}(\Omega\times(0,T))$. For this, we use the growth condition on $K_f$ and \eqref{eq:4th-conv3}. Then, we have
\begin{align*}
 (K_f(b(u)))^2\leq \beta(1+B(u))\in L^{\infty}(0,T;L^1(\Omega)).
\end{align*}
This implies the existence of a constant $C>0$ such that
\begin{align}
 \Vert K_f(b(u)))\Vert_{L^2(\Omega\times(0,T))}\leq C.
 \label{eq:4thh-conv5}
\end{align}
This estimate, the almost everywhere convergence in \eqref{eq:4thh-conv4}, the boundedness of the domain $\Omega\times (0,T)$, and Egorov's theorem imply the weak convergence
\begin{align}
  K_f(b(u^h_M)))\rightharpoonup K_f(b(u))) \quad \text{ in } L^{2}(\Omega\times(0,T)). 
  \label{eq:4th-conv-nonlinear1}
\end{align}
The last step in the proof is to show that
\begin{align*}
 \frac{b(u^h_M(t))-b(u^h_M(t-h))}{h}\rightharpoonup \partial_t b(u) \quad \text{ in } L^{2}(0,T;H_0^{-2}(\Omega)). 
\end{align*}
To do this, we consider the estimate in Lemma \ref{lem:apriori2},
\begin{equation}
\int_{0}^{T}\int_{\Omega}\frac{b(u^h_M(t))-b(u^h_M(t-h))}{h}\phi(t) \,d\textbf{x}\,dt \leq C,
\label{eq:otherform}
\end{equation}
for any $\phi\in L^{2}(0,T;V_{m})$. This uniform estimate implies the existence of a sequence of functionals $v_m^N$ in the dual space $L^2(0,T;V_m^*(\Omega))$ such that
\begin{align}
 \int_{0}^{T} \langle v_m^N,\phi \rangle\,dt= \int_{0}^{T}\int_{\Omega}\frac{b(u^h_M(t))-b(u^h_M(t-h))}{h}\phi \,d\textbf{x}\,dt\leq C.
   \label{eq:4th-weaklimit}
 \end{align}
Hence, there exists a limit $v\in L^2(0,T;H_0^{-2}(\Omega))$ such that
 \begin{align}
  \int_{0}^{T} \langle v_m^N,\phi \rangle\,dt\rightarrow \int_{0}^{T} \langle v,\phi \rangle\,dt
  \label{eq:4th-weaklimit1}
 \end{align}
for all $\phi \in L^2(0,T;V_m(\Omega))$ as $m\rightarrow \infty$ and $h\rightarrow 0$. Since $\bigcup_{m\in\mathbb{N}} V_m$ is dense in $H_0^2(\Omega)$, the convergence result in \eqref{eq:4th-weaklimit1} holds also for all $\phi \in L^2(0,T;H_0^2(\Omega))$. To identify the limit $v$, we consider the test function $\phi \in L^2(0,T;H_0^2(\Omega))$ with $\partial_t \phi\in L^{1}(0,T;L^{\infty}(\Omega))$ and $\phi(t)=0$ for all $t\in (T-h,T]$. Applying summation by parts to the left side of \eqref{eq:otherform} yields
	\begin{align}     
	\int_{0}^{T}\int_{\Omega}&\frac{b(u^h_M(t))-b(u^h_M(t-h))}{h}\phi \,d\textbf{x}\,dt\\
	&=-\frac{1}{h}\int_{-h}^0\int_{\Omega}b(u^h_M)\phi\,d\textbf{x}\,dt-\int_0^T\int_{\Omega} b(u^h_M(t))\frac{\phi(t)-\phi(t-h)}{h}\,d\textbf{x}\,dt,\nonumber\\
	&=\int_0^T\int_{\Omega}\left(b(u_M^{0})-b(u^h_M(t))\right)\frac{\phi(t)-\phi(t-h)}{h} \,d\textbf{x}\,dt,
	\label{eq:integratedEq}
	\end{align}
where we get the last equality using $ \frac{1}{h}\int_{-h}^0 \phi \,dt=-\int_0^T \frac{\phi(t)-\phi(t-h)}{h}\,dt$. Letting $m\rightarrow \infty$ and $h\rightarrow 0$ and using the strong convergence \eqref{eq:4th-conv2}, we have
\begin{align}
  \int_{0}^{T}\int_{\Omega} v \phi \,d\textbf{x}\,dt = \int_{0}^{T}\int_{\Omega}(b(u^0)-b(u)\partial_{t}\phi,
  \label{eq:4th-weaklimit2}
\end{align}
 for all $\phi \in L^2(0,T;H_0^2(\Omega))$ with $\partial_t \phi\in L^{1}(0,T;L^{\infty}(\Omega))$ and $\phi(T)=0$. The right side of \eqref{eq:4th-weaklimit2} corresponds to the definition of the time derivative of $b(u)$ in the distributional sense. Hence, we have $v=\partial_t b(u)$ and we conclude
 \begin{align}
 \frac{b(u^h_M(t))-b(u_M^{h}(t-h))}{h}\rightharpoonup \partial_t b(u) \quad \text{ in } L^{2}(0,T;H_0^{-2}(\Omega)). 
 \label{eq:4th-conv-nonlinear2}
\end{align}

The existence of a function $u\in L^2(0,T;H_0^2(\Omega))$, the convergence results \eqref{eq:4th-conv-nonlinear1}, and \eqref{eq:4th-conv-nonlinear2} imply that equation \eqref{eq:4th-weakdiscreate} convergences as $m\rightarrow \infty$ and $h\rightarrow 0$ to equation \eqref{eq:4th-condition1} for all test function $\phi\in L^2(0,T;H_0^1(\Omega))$. Hence, the function $u$ satisfies the first condition in Theorem \ref{thm:existence}. Clearly, the second condition in Theorem \ref{thm:existence} is also satisfied using equations \eqref{eq:4th-weaklimit2} and \eqref{eq:4th-conv-nonlinear2}.
\end{proof}

\subsection{Uniqueness}
\label{sec:4th-uniqueness}
In this section, we prove the uniqueness of the weak solution of the transformed problem \eqref{eq:linear1}, \eqref{eq:IBC}.
\begin{theorem}
Let Assumption \ref{ass:fourth-order} be satisfied and the transformed saturation $b$ be strictly monotone increasing, i.e. there exists a constant $a>0$ such that 
\[
\min(b'(\cdot))> a >0.
\]
Then, problem \eqref{eq:linear1}, \eqref{eq:IBC} has a unique weak solution that satisfies the properties \eqref{eq:4th-condition1} and \eqref{eq:4th-condition2}.
\label{thm:transformed-unique}
\end{theorem}
\begin{proof}
Assume that $u_1$ and $u_2$ are two weak solutions of problem \eqref{eq:linear1} with the initial and boundary conditions \eqref{eq:IBC} that satisfy the properties \eqref{eq:4th-condition1} and \eqref{eq:4th-condition2}. Define also
\begin{align}
 g \coloneqq b(u_1)-b(u_2). 
 \label{eq:uniq-diff}
\end{align}
Then, property \eqref{eq:4th-condition2} implies that $g\in L^{\infty}(0,T;H_0^{-2}(\Omega))$ and, consequently, we obtain $g\in L^2(0,T;H_0^{-2}(\Omega))$. Thus, Riesz Representation theorem implies the existence of a unique function $w\in L^2(0,T;H_0^{2}(\Omega))$ such that for any time $\tau\in[0,T]$
\begin{align}
\int_0^{\tau} \langle g,\phi \rangle\,dt =\int_0^{\tau} \langle w,\phi \rangle\,dt,
 \label{eq:duality1}
\end{align}
for all $\phi \in L^2(0,T;H_0^{2}(\Omega))$, where
\begin{align}
  \langle w,\phi \rangle \coloneqq \int_{\Omega}\nabla w \cdot \nabla \phi\,d\textbf{x}+\gamma \int_{\Omega}\Delta w\Delta \phi\,d\textbf{x}.
 \label{eq:duality}
 \end{align}
Substituting the solutions $u_1$ and $u_2$ into equation \eqref{eq:4th-condition1}, using the test function $w\in L^2(0,T;H_0^2(\Omega))$, then subtracting the two equations and using \eqref{eq:uniq-diff} gives
\begin{align}
 \int_0^{\tau}\int_{\Omega} \partial_t g \, w \, d\textbf{x}\,dt& +  \int_{0}^{\tau}\int_{\Omega}\big( \nabla (u_1-u_2)\cdot\nabla w+\gamma\Delta (u_1-u_2)\Delta w \big) \, d\textbf{x}\,dt\nonumber\\ &= \int_{0}^{\tau}\int_{\Omega} \big( K_f(b(u_1))- K_f(b(u_2))\big)\textbf{e}_3 \cdot\nabla w\, d\textbf{x}\,dt.
\label{eq:difference-u1-u2}
 \end{align}
Approximating the first term on the left side of \eqref{eq:difference-u1-u2} using backward differences then applying summation by parts yields
 \begin{align}
  \int_{0}^{\tau}&\int_{\Omega}\frac{g(t)-g(t-h)}{h}w(t)\,d\textbf{x}\,dt=- \int_{0}^{\tau}\int_{\Omega}g(t)\frac{w(t)-w(t-h)}{h}\,d\textbf{x}\,dt\nonumber \\&+\frac{1}{h}\int_{\tau-h}^{\tau}\int_{\Omega}g(t)w(t)\,d\textbf{x}\,dt-\frac{1}{h}\int_{-h}^{0}\int_{\Omega}g(t)w(t)\,d\textbf{x}\,dt.
  \label{eq:4th-uniq1}
 \end{align}
Using equations \eqref{eq:duality1} and \eqref{eq:duality}, the first term on the right side of \eqref{eq:4th-uniq1} satisfies
\begin{align*}
 \int_{0}^{\tau}\int_{\Omega}g(t)\frac{w(t)-w(t-h)}{h}\,d\textbf{x}\,dt =&  \int_{0}^{\tau}\int_{\Omega}\nabla w\cdot \frac{\nabla w(t)-\nabla w(t-h)}{h}\,d\textbf{x}\,dt\\&+\gamma \int_{0}^{\tau}\int_{\Omega}\Delta w\, \frac{\Delta w(t)-\Delta w(t-h)}{h}\,d\textbf{x}\,dt.
\end{align*}
Applying summation by parts to the right side of the above equation yields
\begin{align}
 \int_{0}^{\tau}\int_{\Omega}g(t)\frac{w(t)-w(t-h)}{h}\,d\textbf{x}\,dt= &\frac{1}{2h} \int_{\tau-h}^{\tau}\int_{\Omega}|\nabla w|^2+\gamma(\Delta w)^2\,d\textbf{x}\,dt\nonumber\\ & -\frac{1}{2h} \int_{-h}^{0}\int_{\Omega}|\nabla w|^2+\gamma(\Delta w)^2\,d\textbf{x}\,dt.
  \label{eq:4th-uniq2}
\end{align}
The second term on the right side of \eqref{eq:4th-uniq1}, using equations \eqref{eq:duality1} and \eqref{eq:duality}, satisfies
\begin{align}
 \frac{1}{h}\int_{\tau-h}^{\tau}\int_{\Omega}g(t)w(t)\,d\textbf{x}\,dt= \frac{1}{h} \int_{\tau-h}^{\tau}\int_{\Omega}|\nabla w|^2+\gamma(\Delta w)^2\,d\textbf{x}\,dt.
 \label{eq:4th-uniq3}
\end{align}
Similarly, the third term on the right side of \eqref{eq:4th-uniq1} satisfies
\begin{align}
 \frac{1}{h}\int_{-h}^{0}\int_{\Omega}g(t)w(t)\,d\textbf{x}\,dt= \frac{1}{h} \int_{-h}^{0}\int_{\Omega}|\nabla w|^2+\gamma(\Delta w)^2\,d\textbf{x}\,dt.
 \label{eq:4th-uniq4}
\end{align}
Substituting equation \eqref{eq:4th-uniq2}, \eqref{eq:4th-uniq3}, and \eqref{eq:4th-uniq4} into equation \eqref{eq:4th-uniq1} gives
\begin{align}
 \int_{0}^{\tau}\int_{\Omega}\frac{g(t)-g(t-h)}{h}w(t)\,d\textbf{x}\,dt= &\frac{1}{2h} \int_{\tau-h}^{\tau}\int_{\Omega}|\nabla w|^2+\gamma(\Delta w)^2\,d\textbf{x}\,dt\nonumber\\ & -\frac{1}{2h} \int_{-h}^{0}\int_{\Omega}|\nabla w|^2+\gamma(\Delta w)^2\,d\textbf{x}\,dt.
\label{eq:4th-uniq5}
 \end{align}
Using equation \eqref{eq:duality1} and the initial choice \eqref{eq:approx.IC}, the second term on the right side of \eqref{eq:4th-uniq5} satisfies
\begin{align*}
 \int_{-h}^{0}\int_{\Omega}|\nabla w|^2+\gamma(\Delta w)^2\,d\textbf{x}\,dt=\int_{-h}^{0}\int_{\Omega}gw\,d\textbf{x}\,dt =0.
\end{align*}
Hence, letting $h\rightarrow 0$ in equation \eqref{eq:4th-uniq5}, we get that for almost all $\tau\in[0,T]$,
\begin{align}
  \int_0^{\tau}\int_{\Omega} \partial_t g \, w \, d\textbf{x}\,dt= \frac{1}{2} \int_{\Omega}|\nabla w(\tau)|^2+\gamma(\Delta w(\tau))^2\,d\textbf{x}.
  \label{eq:uniqT1}
\end{align}
Using \eqref{eq:duality1} with $\phi=u_1-u_2$, the second term on the left side of \eqref{eq:difference-u1-u2} satisfies
\begin{align}
 \int_{0}^{\tau}&\int_{\Omega}\nabla (u_1-u_2)\cdot\nabla w+\gamma\Delta (u_1-u_2)\Delta w \, d\textbf{x}\,dt\nonumber\\&=\int_{0}^{\tau}\int_{\Omega} (u_1-u_2) g\, d\textbf{x}\,dt= \int_{0}^{\tau}\int_{\Omega} (u_1-u_2) \left( b(u_1)-b(u_2)\right)\, d\textbf{x}\,dt.
 \label{eq:uniqT2}
\end{align}
The Lipschitz continuity of $K_f$ and $b$ imply the existence of a constant $L>0$ such that $\max(b'(\cdot)), \,\max(K_f'(\cdot))\leq L$. Using this property and Cauchy's inequality \eqref{eq:cauchy}, with $\epsilon=\frac{1}{2L^2}$, the first term on the right side of equation \eqref{eq:difference-u1-u2} simplifies to
\begin{align*}
\int_{0}^{\tau}\int_{\Omega}&( K(b(u_1)) -  K(b(u_2)))\textbf{e}_3 \cdot\nabla w\, d\textbf{x}\,dt\nonumber\\\leq &\, L\int_{0}^{\tau}\int_{\Omega}\big|( b(u_1)  - b(u_2))\textbf{e}_3 \cdot\nabla w\big|\, d\textbf{x}\,dt\nonumber \\\leq & \,\frac{1}{2L}\int_{0}^{\tau}\int_{\Omega}\left( b(u_1)-b(u_2)\right)^2 d\textbf{x}\,dt+\frac{L^3}{2}\int_{0}^{\tau}\int_{\Omega}|\nabla w|^2\, d\textbf{x}\,dt\nonumber\\
 \leq &\, \frac{1}{2} \int_{0}^{\tau}\int_{\Omega}\big|\big( b(u_1)-b(u_2)\big)(u_1-u_2)\big| d\textbf{x}\,dt+\frac{L^3}{2}\int_{0}^{\tau}\int_{\Omega}|\nabla w|^2\, d\textbf{x}\,dt.
\end{align*}
As the function $b$ is monotone increasing, it follows that 
\begin{align}
 \int_{0}^{\tau}\int_{\Omega}&( K(b(u_1)) -  K(b(u_2)))\textbf{e}_3 \cdot\nabla w\, d\textbf{x}\,dt\nonumber\\\leq &\,\frac{1}{2} \int_{0}^{\tau}\int_{\Omega}\big( b(u_1)-b(u_2)\big)(u_1-u_2) d\textbf{x}\,dt+\frac{L^3}{2}\int_{0}^{\tau}\int_{\Omega}|\nabla w|^2\, d\textbf{x}\,dt.
 \label{eq:uniqT3}
\end{align}
Substituting \eqref{eq:uniqT1}, \eqref{eq:uniqT2}, and \eqref{eq:uniqT3} into \eqref{eq:difference-u1-u2} yields, for almost all $\tau\in[0,T]$,
\begin{align}
 \dfrac{1}{2}\int_{\Omega}|\nabla w(\tau)|^2\,d\textbf{x}+ \dfrac{1}{2}\int_{\Omega}(\Delta w(\tau))^2\,d\textbf{x}&+\frac{1}{2}\int_{0}^{\tau}\int_{\Omega} \left( b(u_1)-b(u_2)\right) (u_1-u_2)\, d\textbf{x}\,dt\nonumber \\ &\leq\frac{L^3}{2}\int_{0}^{\tau}\int_{\Omega}|\nabla w|^2\, d\textbf{x}\,dt.
 \label{eq:4th-unique}
\end{align}
Since $b$ is a monotone increasing function, the third term on the left side of equation \eqref{eq:4th-unique} is nonnegative. Thus, applying Gronwall's inequality to the first term on the left side gives 
\begin{align}
 \int_{\Omega}|\nabla w(\tau)|^2\,d\textbf{x}=0,
 \label{eq:Gronwall-fourth}
\end{align}
for any $\tau\in[0,T]$. Substituting \eqref{eq:Gronwall-fourth} in equation \eqref{eq:4th-unique} yields
\begin{align}
\int_{0}^{\tau}\int_{\Omega} (b(u_1)-b(u_2)) (u_1-u_2)\, d\textbf{x}\,dt=0.
\label{eq:4th-last}
\end{align}
Using the strict monotonicity of $b$, equation \eqref{eq:4th-last} implies that $u_1=u_2$.
\end{proof}

\section{Regularity}
\label{sec:4th-regularity}
In this section, we improve the regularity of the weak solution from $u\in L^2(0,T;H^2_0(\Omega))$ to $u\in H^1(\Omega\times (0,T))\cap L^2(0,T;H^2_0(\Omega))$. For this, it is sufficient to prove that $\partial_{t} u\in L^{2}(\Omega\times(0,T))$.
	\begin{lemma}
	Let Assumption \ref{ass:fourth-order} be satisfied and the transformed saturation $b$ be strictly monotone increasing, i.e. there exists a constant $a>0$ such that $\min(b'(\cdot))> a >0$. Then, the weak solution $u\in L^2(0,T;H_0^2(\Omega))$ of the transformed problem \eqref{eq:linear1} and \eqref{eq:IBC} satisfies the property that $\partial_{t} u\in L^{2}(\Omega\times(0,T))$.
	\label{lem:regularity}
	\end{lemma}
	\begin{proof}
	Multiplying equation \eqref{eq:approx.Eq} by $\frac{\alpha_{Mi}^h(t)-\alpha_{Mi}^{h}(t-h)}{h}$, summing for $i=1,\cdots,M$, integrating from $0$ to $T$, and using the Gauss theorem yields
	\begin{align}
	&\frac{1}{(h)^2}\int_{0}^{T}\int_{\Omega}\left(b(u^h_M(t))- b(u^h_M(t-h))\right)\big(u^h_M(t) -u^h_M(t-h)\big)\,d\textbf{x}\,dt\nonumber\\ &+\int_{0}^{T}\int_{\Omega}\nabla u^h_M(t)\cdot  \frac{\nabla u^h_M(t)-\nabla u^{h}_M(t-h)}{h}+\gamma \Delta u^h_M(t)  \frac{\Delta u^h_M(t)- \Delta u^{h}_M(t-h)}{h}\,d\textbf{x}\,dt\nonumber\\& =-\int_{0}^{T}\int_{\Omega}\nabla\cdot \left(K_f(b(u^h_M(t))\textbf{e}_3\right) \frac{u^h_M(t)-u^{h}_M(t-h)}{h} \,d\textbf{x}\,dt.
	\label{eq:4-well-1}
	\end{align}
	Using the strict positivity of $b'$, the first term on the left side of \eqref{eq:4-well-1} satisfies
	\begin{align}
	\frac{1}{(h)^2}\int_{h}^{T}\int_{\Omega}\big(b(u^h_M(t))- &b(u^h_M(t-h))\big) \left(u^h_M(t) -u^h_M(t-h)\right)\,d\textbf{x}\,dt\nonumber \\& \geq a\int_0^{T}\int_{\Omega}\left(\frac{ u^h_M(t)-u^{h}_M(t-h)}{h}\right)^{2}\,d\textbf{x}\,dt.
	\label{eq:regularity1}
	\end{align}
	Applying summation by parts to the second term on the left side of \eqref{eq:4-well-1}, we have
	\begin{align*}
	2\int_{0}^{T}\int_{\Omega}\nabla u^h_M\cdot \,& \frac{\nabla u^h_M(t)-\nabla u^{h}_M(t-h)}{h}\,d\textbf{x}\,dt\nonumber\\
	 =&\frac{1}{h} \int_{T-h} ^T \int_{\Omega}|\nabla u^h_M(t)|^2\,d\textbf{x}\,dt - \frac{1}{h} \int_{-h}^0\int_{\Omega} |\nabla u^h_M(t)|^2 \,d\textbf{x}\,dt.
	\end{align*}
	Then, as the discrete solution is a step function in time, we get
	\begin{align}
	 \int_{0}^{T}\int_{\Omega}\nabla u^h_M\cdot \,& \frac{\nabla u^h_M(t)-\nabla u^{h}_M(t-h)}{h}\,d\textbf{x}\,dt= \frac{1}{2} \int_{\Omega}|\nabla u^h_M(T)|^2-|\nabla u^{0}_{M}|^2\,d\textbf{x}.
	  \label{eq:regularity2}
	\end{align}
	Similarly, the third term on the right side of \eqref{eq:4-well-1} simplifies to
	\begin{align}
	 \int_{0}^{T}\int_{\Omega}\Delta u^h_M\,  & \frac{\Delta u^h_M(t)-\Delta u^{h}_M(t-h)}{h}\,d\textbf{x}\,dt= \frac{1}{2} \int_{\Omega}\left(\Delta u^h_M(T)\right)^2-\left(\Delta u^{0}_{M}\right)^2\,d\textbf{x}.
	  \label{eq:regularity3}
	\end{align}
	The Lipschitz continuity of $K_f$ and $b$ implies the existence of a constant $L>0$ such that $\max(b'(\cdot)), \,\max(K_f'(\cdot))\leq L$. Using this propoerty and Cauchy's inequality \eqref{eq:cauchy}, with $\epsilon=\tfrac{L^2}{a}$, the right side of \eqref{eq:4-well-1} gives
	\begin{align}
	  \int_{0}^{T}\int_{\Omega}&\left|\nabla\cdot(K_f(b(u^h_{M})\textbf{e}_3)\,\frac{u^h_M(t)-u^{h}_M(t-h)}{h}\right| \,d\textbf{x}\,dt\nonumber\\ & \leq L^2  \int_{0}^{T}\int_{\Omega}\left|\nabla u^h_M \,\frac{u^h_M(t)-u^{h}_M(t-h)}{h}\right|
	  \,d\textbf{x}\,dt\nonumber\\ &\leq \frac{L^4}{a}\int_{0}^{T}\int_{\Omega}\left|\nabla u^h_M\right|^2 \,d\textbf{x}\,dt+ \frac{a}{4 } \int_{0}^{T}\int_{\Omega}\left(\frac{u^h_M(t)-u^{h}_M(t-h)}{h}\right)^2\,d\textbf{x}\,dt.
	  \label{eq:regularity4}
	\end{align}
	Substituting \eqref{eq:regularity1}, \eqref{eq:regularity2}, \eqref{eq:regularity3}, and \eqref{eq:regularity4} into inequality \eqref{eq:4-well-1} gives
	\begin{align*}
	  \frac{3a}{4}\int_0^{T}\int_{\Omega}&\left(\frac{ u^h_M(t)-u^{h}_M(t-h)}{h}\right)^{2}\,d\textbf{x}\,dt+\frac{1}{2}\int_{\Omega}\vert\nabla u^h_M(T)\vert^{2}+\gamma\left(\Delta u^h_M(T)\right)^{2}\,d\textbf{x}\\\leq&  \dfrac{1}{2}\int_{\Omega} |\nabla u_M^0|^2+ \gamma(\Delta u_M^0)^2\,d\textbf{x}+\frac{L^4}{a}\int_{0}^{T}\int_{\Omega}\left|\nabla u^h_M\right|^2 \,d\textbf{x}\,dt.
	\end{align*}
	Then, Lemma \ref{lem:apriori1} implies the existence of a constant $c>0$ such that
	\begin{align*}
	 \int_0^{T}\int_{\Omega}\left(\frac{ u^h_M(t)-u^{h}_M(t-h)}{h}\right)^{2}\,d\textbf{x}\,dt\leq c.
	\end{align*}
	This uniform estimate implies that, up to a subsequence,
	\begin{align}
	 \frac{ u^h_M(t)-u^{h}_M(t-h)}{h} \rightharpoonup  \partial_t u \quad \text{ in } L^2(\Omega\times (0,T)).
	\end{align}
	\end{proof}
	 
	 \begin{corollary}
	   Let the assumptions of Theorem \ref{thm:transformed-unique} be satisfied. Then, we have the strong convergence
	   \begin{align*}
	    u^h_M \rightarrow u \quad \text{ in } L^2(\Omega\times(0,T)).
	   \end{align*}
	 \end{corollary}
	 \begin{proof}
	  The proof follows using the estimates in Lemma \ref{lem:apriori1} and \ref{lem:regularity} together with Rellich Kondrachov Compactness theorem with dimension $n=4$ of the domain $\Omega\times(0,T)$.
	 \end{proof}

	 \begin{corollary}
	  Let the assumptions of Theorem \ref{thm:transformed-unique} be satisfied. Then, the transformed saturation $b$ satisfies
	  \begin{align*}
	   b(u) \in&\, C([0,T];L^2(\Omega)),
	  \end{align*}
	  and the initial condition satisfies
	  \begin{align*}
	   b(u(0))=b(u^0) \quad\text{almost everywhere.}
	  \end{align*}
	 \end{corollary}
	  \begin{proof}
	   The Lipschitz continuity of the transformed saturation $b$ and Lemma \ref{lem:regularity} imply
	   \begin{align}
	    \partial_t b(u)=b'(u) \partial_t u  \in L^2(\Omega\times(0,T)). 
	   \end{align}
	    This yields also that
	  \begin{align}
	    b(u)\in C([0,T];L^2(\Omega)).
	  \end{align}
	  To prove that $b(u(0))=b(u^0)$ almost everywhere, we choose a test function $\phi\in C^1([0,T],H_0^2(\Omega))$ in equation \eqref{eq:4th-condition1} such that $\phi(T)=0$. Then, Gauss theorem gives
	\begin{align}
	\int_{0}^{T}\int_{\Omega} \Bigl(b(u)\partial_{t}\phi- K_f(b(u))\textbf{e}_3 \cdot\nabla\phi + \nabla u \cdot\nabla\phi\Bigr. & \Bigl.+\gamma\Delta u\Delta \phi\Bigr) \, d\textbf{x}\,dt\nonumber\\&=\int_{\Omega} b(u(0))\phi(0) \, d\textbf{x}.
	\label{eq:4th-weak1}
	\end{align}
Applying summation by parts to the first term in equation \eqref{eq:4th-weakdiscreate} yields
\begin{align}
\dfrac{1}{h} \int_{0}^{\tau} \int_{\Omega} & b(u^h_M(t))\big(\phi(t)-\phi(t-h) \big)\,d\textbf{x}\,dt+\int_{0}^{\tau}\int_{\Omega}\nabla u^h_M\cdot\nabla\phi\,d\textbf{x}\,dt\nonumber\\&+\gamma\int_{0}^{\tau}\int_{\Omega}\Delta u^h_M\Delta \phi\,d\textbf{x}\,dt - \int_{0}^{\tau}\int_{\Omega}K_f(b(u^h_M)\textbf{e}_3\cdot \nabla \phi \,d\textbf{x}\,dt\nonumber\\&= \int_{\Omega} b(u_M^0)\phi(0)\,dx\,dz.
\label{eq:4th-initial}
\end{align}
Letting $M\rightarrow \infty$ and $h\rightarrow 0$ in equation \eqref{eq:4th-initial} yields, up to a subsequence, that
\begin{align}
\int_{0}^{\tau} \int_{\Omega} &\partial_t\phi\, b(u(t))\,d\textbf{x}\,dt+\int_{0}^{\tau}\int_{\Omega}\nabla u\cdot\nabla\phi\,d\textbf{x}\,dt+\gamma\int_{0}^{\tau}\int_{\Omega}\Delta u\Delta \phi\,d\textbf{x}\,dt \nonumber\\&- \int_{0}^{\tau}\int_{\Omega}K_f(b(u)\textbf{e}_3\cdot \nabla \phi \,d\textbf{x}\,dt= \int_{\Omega} b(u^0)\phi(0)\,dx\,dz.
\label{eq:4th-weak2}
\end{align}
since $u_M^0\rightarrow u^0$ in $L^2(\Omega)$ as $M\rightarrow \infty$. As $\phi(0)$ is arbitrarily chosen, comparing equation \eqref{eq:4th-weak1} and \eqref{eq:4th-weak2} yields that $b(u(0))=b(u^0)$ almost everywhere.
	  \end{proof}

\section{Well-posedness of the Fourth-Order Model}
\label{sec:4th-original}
In this section, we utilize the well-posedness of the transformed problem \eqref{eq:linear1} and \eqref{eq:IBC} to prove the well-posedness of the fourth-order model \eqref{eq:Armiti2} and \eqref{eq:4th-IBC}. For this, we stress that the coefficients $S,\,S',\,K_f$ are strictly positive. Then, we apply the inverse of Kirchhoff's transformation to the weak solution of the transformed problem \eqref{eq:linear1} and \eqref{eq:IBC}.

\begin{definition}
 Let the function $S=S(p)$ be Lipschitz continuous and $K_f=K_f(S(p))$ be bounded. We call $p\in H^1(\Omega\times(0,T))$ a weak solution of the fourth-order problem \eqref{eq:Armiti2} and \eqref{eq:4th-IBC} if it satisfies the conditions
 \begin{enumerate}
	\item $\partial_t S(p)\in L^{2}(\Omega\times (0,T))$ and $\nabla\cdot\left(K_f(S(p))\nabla p\right)\in L^{2}(\Omega\times (0,T))$ such that\emph{
	\begin{align*}
	\int_{0}^{T}\int_{\Omega} \Bigl(\partial_{t}S(p)\phi - K_f(S(p))\textbf{e}_3 \cdot\nabla\phi&+ K_f(S(p))\nabla p \cdot\nabla\phi\Bigr) \, d\textbf{x}\,dt\nonumber \\ &+\gamma\int_{0}^{T}\int_{\Omega}  \nabla\cdot\bigl(K_f(S(p))\nabla p\bigr)\Delta \phi \, d\textbf{x}\,dt =0,
	\end{align*}}
	for every test function $\phi \in L^{2}(0,T;H_{0}^{2}(\Omega))$.
	\item $S\bigl(p(0)\bigr)=S(p^0)$ almost everywhere. 
	\end{enumerate}
	\label{def:weaksolution2}
\end{definition}
\begin{theorem}
Assume that the initial condition in \eqref{eq:4th-IBC} satisfies $p^0\in H_0^2(\Omega)$ and the saturation function $S\in C^1(\mathbb{R})$ is Lipschitz continuity, strictly positive, and strictly monotone increasing. Assume also that the conductivity function $K_f\in C^1(\mathbb{R})$ is strictly positive, bounded, and monotone increasing. Let $u\in H^1(\Omega\times (0,T))\cap L^2(0,T;H_0^2(\Omega))$ be the weak solution of the transformed problem \eqref{eq:linear1} and \eqref{eq:IBC}. Then $p=\psi^{-1}(u)$, where $\psi$ is Kirchhoff's transformation, is the unique weak solution of the fourth-order problem \eqref{eq:Armiti2} and \eqref{eq:4th-IBC} according to Definition \ref{def:weaksolution2}.
 \label{thm:wellposed-original}
\end{theorem}

\begin{proof}
Using equation \eqref{eq:chainrule}, Lemma \ref{lem:regularity}, the boundedness and the strict positivity of $K_f$, we have
	\begin{align}
	\begin{array}{cc}
	 \nabla p=& \dfrac{\nabla u}{K_f\bigl(S(\psi^{-1}(u))\bigr)}\quad \in L^{2}(\Omega\times(0,T)),\\
 	\partial_t p=& \dfrac{\partial_t u}{K_f \bigl(S(\psi^{-1}(u))\bigr)}\quad \in L^{2}(\Omega\times(0,T)),
	\end{array}
	  \label{eq:4th-conv-transform}
	\end{align}
where there exists a constant $\delta>0$ such that $K_f>\delta$. These estimeates and Poincar\'{e}'s inequality implies that $p\in H^1(\Omega\times(0,T))$. Consequently, we have
\begin{align}
	    p\in C([0,T];L^2(\Omega)).
	  \end{align}
In addition to this, we have
\begin{align}
\begin{array}{rll}
    \nabla\cdot\bigl(K_f(S(p))\nabla p\bigr)=& \Delta u & \quad\in L^{2}(\Omega\times(0,T)).
\end{array}
\label{eq:result}
\end{align}
The Lipschitz continuity of the saturation $S$ and the second equation in \eqref{eq:4th-conv-transform} imply 
\begin{align*}
 \partial_t S(p)=S'(p)\partial_t p \in L^{2}(\Omega\times(0,T)).
\end{align*}

These estimates imply that $p$ satisfies the conditions in Definition \ref{def:weaksolution2} and, thus, is a weak solution of the fourth-order model \eqref{eq:Armiti2} and \eqref{eq:4th-IBC}. In the same way, if $p\in H^1(\Omega\times(0,T))$ is a weak solution of the fourth-order problem \eqref{eq:Armiti2} and \eqref{eq:4th-IBC} as in Definition \ref{def:weaksolution2}, then the Kirchhoff-transformed $u=\psi(p) \in L^2(0,T;H_0^2(\Omega))$ is a weak solution of the transformed fourth-order problem \eqref{eq:linear1} and \eqref{eq:IBC}. This implies that the fourth-order problem \eqref{eq:Armiti2}, \eqref{eq:4th-IBC} and the transformed fourth-order problem \eqref{eq:linear1} and \eqref{eq:IBC} are equivalent. This equivalency, the uniqueness of the weak solution $u$ of the transformed problem by Theorem \ref{thm:transformed-unique}, and the strict monotonicity of Kirchhoff's transformation imply the uniqueness of the weak solution $p$ of the fourth-order problem \eqref{eq:Armiti2} and \eqref{eq:4th-IBC}. 
\end{proof}

\bibliographystyle{plain}
\bibliography{fourth-order}

\end{document}